\documentclass[10pt,letter,reqno]{amsart}
\usepackage{amsmath, amsthm, amscd, amsfonts, amssymb, graphicx, color}
\usepackage[bookmarksnumbered, colorlinks, plainpages]{hyperref}
\hypersetup{colorlinks=true,linkcolor=red, anchorcolor=green, citecolor=cyan, urlcolor=red, filecolor=magenta, pdftoolbar=true}
\usepackage[all]{xy}

\usepackage{geometry}
\newcommand{\kk}             {{\mathbb K}}
\newcommand{\into}          {\rightarrow}
\newcommand{\tens}         {\otimes}
\newcommand{\fhi}            {\varphi}
\newcommand{\so}            {\Sigma\mbox{-ope\-ra\-tor}}
\newcommand{\sos}           {\Sigma\mbox{-ope\-ra\-tors}}
\newcommand{\stn}           {\Sigma\mbox{-tensor norm}}
\newcommand{\sumim}      {\sum\limits_{i=1}^m}

\newcommand{\finn}          {\mathcal{FIN}}
\newcommand{\ban}          {\mathcal{BAN}}
\newcommand{\lev}            {\left\langle}
\newcommand{\rev}            {\right\rangle}
\newcommand{\varfhi}        {\overline{\fhi}}

\newcommand{\br}    		    {{\beta|}}
\newcommand{\ab}    		{\alpha^\beta}
\newcommand{\abr}    		{\alpha^\br}
\newcommand{\vb}   		    {\nu_\beta}
\newcommand{\vt}     		{\nu_\theta}
\newcommand{\vbr}   		{\nu_{\br}}

\newcommand{\xxx}           {X_1,\dots, X_n}
\newcommand{\zzzn}            {Z_1,\dots, Z_n}
\newcommand{\zzzm}            {Z_1,\dots, Z_m}
\newcommand{\eee}          {E_1,\dots, E_n}
\newcommand{\mmm}          {M_1,\dots, M_n}
\newcommand{\xxy}           {X_1,\dots ,X_n,Y}

\newcommand{\eef}           {E_1,\dots ,E_n,F}
\newcommand{\xxyb}         {\left(n,\xxy,\beta\right)}

\newcommand{\eefb}         {\left(n,\eef,\beta\right)}
\newcommand{\xpx}           {X_1\times\dots\times X_n}

\newcommand{\zpzm}           {Z_1\times\dots\times Z_m}
\newcommand{\epe}           {E_1\times\dots\times E_n}
\newcommand{\xxp}           {(x^1,\dots ,x^n)}

\newcommand{\sxx}           {\Sigma_{X_1,\dots, X_n}}
\newcommand{\sxxb}         {\Sigma_{X_1, \dots, X_n}^\beta}

\newcommand{\seebr}        {\Sigma_{E_1, \dots, E_n}^{\br}}

\newcommand{\szzm}            {\Sigma_{Z_1,\dots, Z_m}}

\newcommand{\xtx}      		{X_1\tens\cdots\tens X_n}
\newcommand{\ztz}      		{Z_1\tens\cdots\tens Z_n}
\newcommand{\ztzm}      		{Z_1\tens\cdots\tens Z_m}
\newcommand{\ete}      		{E_1\tens\cdots\tens E_n}

\newcommand{\xtxb}    		{\left(\xtx,\beta\right)}
\newcommand{\ztzt}    		{\left(\ztz,\theta\right)}
\newcommand{\ztzmt}    		{\left(\ztzm,\theta\right)}
\newcommand{\xtxty}    	  	{X_1\tens\cdots\tens X_n\tens Y}
\newcommand{\etetf}        {E_1\tens\cdots \tens E_n\tens F}

\newcommand{\xtxtyd}   	{X_1\tens\cdots\tens X_n\tens Y^*}
\newcommand{\xtxtydd}  	{X_1\tens\cdots\tens X_n\tens Y^{**}}

\newcommand{\etetfd}      {E_1\tens\cdots \tens E_n\tens F^*}

\newcommand{\xtxp}         {X_1\hat{\tens}_\pi\cdots\hat{\tens}_\pi X_n}
\newcommand{\xxt}           {x^1\tens\cdots\tens  x^n}
\newcommand{\xxyt}         {x^1\tens\cdots \tens x^n\tens y}

\newcommand{\pmq}         {p-q}
\newcommand{\pimqi}       {p_i-q_i}

\newcommand{\Txxy}         {T:\xpx\into Y}
\newcommand{\tlin}           {\widetilde{T}}
\newcommand{\ft}              {f_T}
\newcommand{\rlin}           {\widetilde{R}}
\newcommand{\fr}              {f_R}

\newcommand{\Lxxy}           {\mathcal{L}(\xxx; Y)}
\newcommand{\Lbxx}         {\mathcal{L}^\beta\left(\xxx\right)}
\newcommand{\Lbxxy}       {\mathcal{L}^\beta\left(\xxx;Y\right)}

\newcommand{\Lbee}        {\mathcal{L}^\beta\left(\eee\right)}
\newcommand{\Lbree}       {\mathcal{L}^{\br}\left(\eee\right)}

\newcommand{\Ltzzm}          {\mathcal{L}^\theta\left(\zzzm\right)}
\newcommand{\Lb}            {Lip^\beta}

\newcommand{\ideal}         {[\mathcal{A},A]}
\newcommand{\Abxxy}       {\mathcal{A}^\beta(\xxx; Y)}

\newcommand{\Atzzmw}       {\mathcal{A}^\theta(\zzzm; W)}
\newcommand{\Abeef}       {\mathcal{A}^\beta(\eee; F)}
\newcommand{\Abxxyd}     {\mathcal{A}^\beta(\xxx; Y^*)}
\newcommand{\Abeefd}     {\mathcal{A}^\beta(\eee; F^*)}
\newcommand{\Abxxydd}   {\mathcal{A}^\beta(\xxx; Y^{**})}
\newcommand{\Abxxyddd}   {\mathcal{A}^\beta(\xxx; Y^{***})}

\newcommand{\Fbxxy}       {\mathcal{F}^\beta(\xxx; Y)}

\newcommand{\Ab}          {A^\beta}
\newcommand{\At}          {A^\theta}

\newtheorem{theorem}{Theorem}[section]

\newtheorem{proposition}[theorem]{Proposition}
\newtheorem{corollary}[theorem]{Corollary}
\theoremstyle{definition}
\newtheorem{definition}[theorem]{Definition}
\newtheorem{example}[theorem]{Example}

\theoremstyle{remark}
\newtheorem{remark}[theorem]{Remark}
\numberwithin{equation}{section}

\begin{document}

\setcounter{page}{1}

\title[duality between ideals of multilinear operators and tensor norms]{The duality between ideals of multilinear operators and tensor norms}

\author[S. Garc\'ia-Hern\'andez]{Samuel Garc\'ia-Hern\'andez}

\address{Centro de investigaci\'on en Matem\'aticas, Guanajuato M\'exico}
\email{\textcolor[rgb]{0.00,0.00,0.84}{samuelg@cimat.mx}; orcid:0000-0003-4562-299X}

\subjclass[2010]{47L22, 46M05, 46G25, 46B28, 47B10.}








\keywords{Ideals of multilinear operators, tensor norms, duality, operator ideals, Banach spaces}

\begin{abstract}
We develop the duality theory between ideals of multilinear operators and tensor norms that arises from the geometric approach of $\Sigma$-operators. To this end, we introduce and develop the notions of $\Sigma$-ideals of multilinear operators and $\Sigma$-tensor norms. We establish the foundations of this theory by proving a representation theorem for maximal $\Sigma$-ideals of multilinear operators by finitely generated $\Sigma$-tensor norms and a duality theorem for $\Sigma$-tensor norms. For these notions we also develop basic theory and present concrete examples.
\end{abstract}
\maketitle

\section{Introduction}

The theory of operator ideals, mainly developed by Pietsch in \cite{pietsch78}, is a branch of Functional Analysis that provides a systematic framework to study linear operators grouping them according to the so called ideal property. Many examples of operator ideals find applications in a wide variety of topics related to Functional Analysis, for instance, Operator, Probability and Measure theories as well as Harmonic Analysis and geometry of Banach spaces. Some of these applications are shown throughout the monographs \cite{defant93, diestel95, pietsch78, ryan02, tomczak89,wojtaszczyk91}.

The theory of operator ideals is closely related with the theory of tensor norms developed by Grothendieck \cite{grothendieck53}. Indeed, if we restrict our attention to finite dimensional normed spaces, there exists a one to one relation between operator ideals and tensor norms; this is summarized by $E^*\tens_{\nu} F=\mathcal{A}(E;F)$ where the equality stands for an isometric isomorphism. This relation does not hold for general Banach spaces; however, it is possible to extend it under certain assumptions. To be specific, maximal ideals of linear operators are represented by finitely generated tensor norms through the Representation Theorem for Maximal Ideals. This result is a keystone in the matter since it provides a bridge between operator ideals and tensor norms phenomena in the general case of Banach spaces. This time, the interplay between this two fields has the form $(X\tens_\alpha Y)^*=\mathcal{A}(X;Y^*)$. For a detailed exposition of this result we strongly recommend \cite[Section 17.5]{defant93}.

The purpose of this article is to develop the duality theory for ideals of multilinear operators and tensor norms based on the geometric approach of $\Sigma$-operators. Our two main results are a representation and a duality theorem. These results are situated in the research developed in recent years and devoted to extend the theory of linear operators to other nonlinear situations. Among other perspectives, this research consists of extending ideals of linear operators to the multilinear setting. It is noteworthy the nonexistence of a canonical path to be followed to extend the theory of linear operators to the multilinear setting. For this reason, some classes of linear operators have more than one multilinear version, for instance, the cases of absolutely summing and nuclear operators. A recent approach, detailed in \cite{fernandez-unzueta18}, proposes to study multilinear operators based on the idea of considering multilinear mappings as homogeneous functions. This approach gave rise to the notion of $\sos$. In \cite{angulo18, fernandez-unzueta18, fernandez-unzueta18a, fernandez-unzueta18b} the reader can appreciate the results of this way of thinking. These references develop the classes of bounded multilinear operators, multilinear operators factoring through Hilbert spaces, Lipschitz $p$-summing and $(p,q)$-dominated multilinear operators under the perspective of $\Sigma$-operators. Each of these classes seems to work properly since they verify analogues of the fundamental features of their linear counterparts. Even more, these classes are closed under certain compositions and have a peculiar maximal ideal behavior. Furthermore, they are in duality relation with particular tensor norms which enjoy of a peculiar finitely generated behavior. These duality relations along with the relation between operator ideals and tensor norms lead, naturally,  to investigate a proper notion of ideals of multilinear operators and tensor norms following the geometric approach of $\sos$. For achieving our goals we introduce a proper notion of ideals of multilinear operators and tensor norms. We also specify local properties for these concepts, that is, properties that ensure a good behavior in general Banach spaces regarding finite dimensional subspaces. Specifically, we introduce $\Sigma$-ideals of multilinear operators and $\Sigma$-tensor norms. For these concepts we prove a Representation Theorem for Maximal $\Sigma$-ideals and a Duality Theorem for $\Sigma$-tensor norms.

Let us we describe our results. Fix a positive integer $n$ and Banach spaces $\xxx$ and $Y$. Every bounded multilinear operator $\Txxy$ has associated a homogeneous function $f_T:\sxx\into Y$ where $\sxx\subset \xtx$ is the so called Segre cone of the Banach spaces $X_i$. The projective tensor norm $\pi$ defines a metric structure on $\sxx$. The continuity of the mapping $T$ is equivalent to the Lipschitz condition of $f_T$. A function $f_T$ as above is referred to as the bounded $\Sigma$-operator associated to $T$. For a better development of the theory, particularly on the local features, we have to consider not just the norm $\pi$ but all reasonable crossnorms $\beta$ on $\xtx$.

The notion of $\Sigma$-ideal is intended to group multilinear operators taking into account their associated $\Sigma$-operators (see \eqref{idealproperty}). A component in a $\Sigma$-ideal $\ideal$ is denoted by $\Abxxy$ and is determined by an arrangement of the form $(n,\xxx,Y,\beta)$ where $n$ is a positive integer, $\xxx$ and $Y$ are Banach spaces and $\beta$ is a reasonable crossnorm on $\xtx$.  By definition, the component $\Abxxy$ is a linear subspace of $\Lxxy$ and is normed with its ideal norm denoted by $\Ab$. Roughly speaking, $\mathcal{A}$ is the union of all the possible components and $A$ is an application whose restriction to each component  agrees with the norm $\Ab$ (see Definition \ref{ideal}).

The idea of representing $\Sigma$-ideals in tensorial terms gives rise to the notion of $\Sigma$-tensor norms. These norms appear in two different versions: \emph{on spaces} and \emph{on duals}. Roughly speaking, a $\Sigma$-tensor norm on spaces $\alpha$ assigns a norm $\ab$ on $\xtxty$ such that the dual space $(\xtxty,\ab)^*$ to be isometrically isomorphic to a component $\Abxxyd$ for a certain $\Sigma$-ideal $\ideal$. For a better development of the theory we are in the necessity of consider tensor products of the form $\Lbxxy\tens Y$, where $\Lbxxy$ stand for the dual space of $\xtxb$. A $\Sigma$-tensor norm on duals assigns a norm $\vb$ on $\Lbxxy\tens Y$. As we will see, these two versions are dual to each other (see Definition \ref{tns}, Definition \ref{tnd}  and Theorem \ref{tnstotnd}).

Each of the concepts we are presenting defines the two others, uniquely, in the class of finite dimensional normed spaces. This fact is summarized by the mappings
\begin{eqnarray}\label{relationfinitedimensions}
(\Lbee\tens F^*,\vb)=(\etetf,\ab)^*=\Abeefd\\
\fhi\tens y^*\hspace{1cm}\mapsto\hspace{1.4cm}\fhi\tens y^*\hspace{1.3cm}\mapsto\hspace{1.1cm} \fhi\cdot y^*\hspace{1cm}\nonumber
\end{eqnarray}
where $E_i$ and $F$ are finite dimensional normed spaces, $\nu$ is a $\Sigma$-tensor norm on duals, $\alpha$ is a $\Sigma$-tensor norm on spaces, $\ideal$ is a $\Sigma$-ideal, the tensor $\fhi\tens y^*$ acts by evaluation on $\etetf$ and the operator $\fhi\cdot y^*$ is defined by $\fhi\cdot y^*(z)=\fhi(z)y^*$ for all $z$ in $\epe$. Here, the equality symbol stands for an isometric isomorphism (see Theorems \ref{tnstotnd} and \ref{tndsideal}).

In order to extend \eqref{relationfinitedimensions} to general Banach spaces we are in the necessity of defining maximal $\Sigma$-ideals, finitely generated $\Sigma$-tensor norms on spaces and cofinitely generated $\Sigma$-tensor norms on duals. Our main results are expressed in these terms. Theorem \ref{rt} is The Representation Theorem for Maximal $\Sigma$-ideals and affirms that
\begin{align*}
\left(\xtxty,\ab\right)^*&=\Abxxyd\\
\left(\xtxtyd,\ab\right)^*\cap\mathcal{L}(\xxx;Y) &=\Abxxy
\end{align*}
hold linearly and isometrically if $\alpha$ is a finitely generated $\Sigma$-tensor norm on spaces related with the maximal  $\Sigma$-ideal $\ideal$ by \eqref{relationfinitedimensions}. Theorem \ref{dt} is The Duality  Theorem for $\Sigma$-tensor norms and affirms that the linear inclusion
\begin{eqnarray*}
\left(\Lbxx\tens Y,\vb\right)\hookrightarrow\left(\xtxtyd,\ab\right)^*,
\end{eqnarray*}
given by evaluation, is isometric for all Banach spaces if $\nu$ is cofinitely generated, $\alpha$ is finitely generated, and they are related in finite dimensions as in \eqref{relationfinitedimensions}.

Throughout this article we present, as examples of $\Sigma$-ideals, multilinear versions of the most classical examples of operator ideals. On the side of tensor norms, we exhibit the greatest and least $\Sigma$-tensor norms in both versions. We also recall from \cite{fernandez-unzueta18a} the Laprest\'e-type $\Sigma$-tensor norms on spaces.

It is noteworthy the consistency of the results presented here with those of \cite{floret_hunfeld02} and \cite{achour16, cabrera-padilla15, cabrera-padilla16a, cabrera-padilla17} where other relations of ideals and tensor norms are developed for multilinear and Lipschitz mappings.

The distribution of the material is the next. In Section 1.1 we fix notation. In Section 2 we recall basic theory of $\Sigma$-operators. In Section 3 we introduce and develop the notion of $\Sigma$-ideals. In Section 4 we introduce the notion of $\Sigma$-tensor norms. In this section we also specify the relation between $\Sigma$-tensor norms and $\Sigma$-ideals in finite dimensions. In Section 5 we relate $\Sigma$-tensor norms and $\Sigma$-ideals in general Banach spaces. In this section we prove the representation and duality theorems.

\subsection{Notation}

The letter $\kk$ denotes the filed of real or complex numbers. The unit ball of the normed space $X$ is denoted by $B_X$. The operator $K_X:X\into X^{**} $ denote the canonical embedding given by evaluation. The set of finite dimensional subspaces of $X$ is denoted by $\mathcal{F}(X)$ and the set of finite codimensional subspaces is denoted by $\mathcal{CF}(X)$. The class of all Banach spaces is denoted by $\ban$ while $\finn$ denote the class of all finite dimensional normed spaces. If $f:A\into B$ is a function, then we use the notation $\lev f  ,  a \rev:=f(a)$.

Throughout this article $n$ denotes a positive integer and the letters $X_1,\dots, X_n$, $Z_1,\dots, Z_n$, $Y$, $W$ denote Banach spaces. The symbol $\Lxxy$ denotes the Banach space of all bounded multilinear operators $T:\xpx\into Y$ with the usual norm $\|T\|=\sup\{ \|T\xxp\| \;|\; x_i\in B_{X_i}, 1\leq i\leq n\}$. If $Y=\kk$ we simply write $\mathcal{L}(\xxx)$.

\section{$\Sigma$-operators and Multilinear Operators}

The set of decomposable tensors of the algebraic tensor product $\xtx$ is denoted by $\sxx$. That is, $\sxx:=\left\{\;  \xxt  \;|\;  x^i\in X_i, 1 \leq i \leq n  \;\right\}$.

From now on $\beta$ denote a reasonable crossnorm on $\xtx$. That is, a norm such that $\varepsilon (u)\leq \beta(u) \leq \pi(u)$ for all $u$ in $\xtx$, where $\varepsilon$ and $\pi$ are the injective and projective tensor norms given by
\begin{align*}
\pi(u)&= \inf \left\{\;  \sumim \|x_i^1\|\cdots\|x_i^n\|  \;\Big|\; u=\sumim x_i^1\tens\cdots\tens x_i^n \;\right\},\\
\varepsilon(u)&= \sup \left\{\;  |x_1^*\tens\cdots\tens x_n^*(u)|  \;\Big|\; x_i^*\in B_{X_i^*}, 1\leq i\leq n \;\right\}
\end{align*}
for all $u$ in $\xtx$.

Every norm $\beta$ defines a metric on the set $\sxx$ but according to \cite[Theorem 2.1]{fernandez-unzueta18} all these metrics are equivalent. Explicitly,
\begin{equation}\label{lipschitzequivalent}
\pi(\pmq) \leq 2^{n-1} \beta(\pmq) \qquad \forall\, p,q\in \sxx.
\end{equation}
We also use the symbol $\sxx$ to denote the metric space resulting of restricting the norm $\pi$ to the set of decomposable tensors of $\xtx$. This metric space is the so called Segre cone of $\xxx$.

\begin{remark}\label{betarestricted}
It is straightforward that if $E_i$ is a closed subspace of $X_i$, $1 \leq  i \leq n$, the restriction of the norm $\beta$ to the tensor product $\ete$ is a reasonable crossnorm. Such restrictions are denoted by $\br$. The consideration of these restrictions will be crucial for defining the local behavior of $\Sigma$-ideals and $\Sigma$-tensor norms (see Definition \ref{maximalhull}, Definition \ref{finitelygenerated} and Definition \ref{cofinitelygenerated}).
\end{remark}

The symbol $\Lbxx$ denote the Banach space of all multilinear forms $\fhi:\xpx\into \kk$ whose linearization $\tilde{\fhi}:\xtxb\into \kk$ is bounded provided with the norm $\|\fhi\|_\beta:=\|\tilde{\fhi}\|$. Plainly, the mapping $\fhi\mapsto\tilde{\fhi}$ defines a linear isometry between $\Lbxx$ and $\xtxb^*$.

The universal property of the projective tensor product establishes that for every bounded  multilinear operator $T:\xpx\into Y$ there exists a unique bounded linear operator $\tlin:\xtxp\into Y$ such that $T\xxp=\tlin(\xxt)$ for all $x^i\in X_i$, $1 \leq i \leq n$. In particular, the restriction $\tlin|_{\sxx}:\sxx\into Y$ is a homogeneous Lipschitz function. In this situation, the operator $\tlin$ is called the linearization of $T$ and the function $f_T:=\tlin|_{\sxx}$ is named the bounded $\so$ associated to $T$. Under these circumstances \eqref{lipschitzequivalent} implies
\begin{equation}\label{lipnorm}
\|T\|=Lip^\pi(f_T)\leq\Lb(f_T)\leq 2^{n-1}Lip^\pi(f_T)
\end{equation}
for all bounded multilinear operators $T$ and reasonable crossnorms $\beta$.


\section{ $\Sigma$-Ideals of Multilinear Operators}

We begin our study extending the definition of linear operator ideals (see \cite[Section 9]{defant93}) to the multilinear case following the approach of $\Sigma$-operators. Specifically, we introduce the notion of $\Sigma$-ideal of multilinear operators. In a few words, this new concept is intended to group multilinear operators taking into account their associated $\Sigma$-operators. For this reason, all the properties defining $\Sigma$-ideals require the language of $\Sigma$-operators. For a better exposition, before presenting each new definition we expose necessary details and notation.

An arrangement of the form $\xxyb$ where $n$ is a positive integer, $X_i$, $1 \leq i \leq n$, $Y$ are Banach spaces and $\beta$ is a reasonable crossnorm on $\xtx$ is named an \emph{election in the class} $\ban$. Obvious definition follows for elections in the class of finite dimensional normed spaces $\finn$.

The essential feature of a linear operator ideal is the ideal property. That is, the property which ensures that a composition $RTS$ of linear operators is in a given ideal whenever $T$ is. The first step in our development is to specify the corresponding ideal property. To this end, consider two elections $\xxyb$ and $(m,\zzzm,W)$ in the class $\ban$, a bounded multilinear operator $\Txxy$, a bounded linear operator $S:Y\into W$ and a bounded multilinear operator $R:\zpzm\into \xtxb$ such that $\rlin:(\ztzm,\theta)\into \xtxb$ is bounded and has image contained in $\sxx$. With the help of the diagram
\begin{equation}\label{idealproperty}
\begin{array}{c}
\xymatrix{
\zpzm\ar[d]\ar[drr]^-R                     &&\xpx\ar[d]\ar[drr]^-T && &\\
\szzm\ar[d]\ar[rr]^{\fr}     &&\sxx\ar[d]\ar[rr]^{f_T}  &&Y\ar[r]^S   &W\\
    (\ztzm,\theta)\ar[rr]^-{\rlin}                     &&\xtxb  &&&
}
\end{array}
\end{equation}
it is plain that $Sf_TR=Sf_Tf_R\tens$, where $\tens:\zpzm\into \ztzm$ is the natural multilinear mapping. Thus, instead of studying a composition from $\zpzm$ to $W$ we study $Sf_Tf_R$ taking advantage of the geometry that the normed spaces $(\ztzm,\theta)$ and $\xtxb$ has to offer. A $\Sigma$-operator $f_R$ as above is named $\mathbf{\Sigma}$-$\mathbf{\beta}$-$\mathbf{\theta}$-{\bf operator}.

\begin{definition}\label{ideal}
A {\bf $\mathbf{\Sigma}$-ideal of multilinear operators} $\ideal$ defined on $\ban$ assigns to each election $\xxyb$ in the class of Banach spaces $\ban$ a linear subspace, named component, $\Abxxy$ of $\Lxxy$ and a norm $\Ab$ on $\Abxxy$ which makes it a Banach space and satisfy the following properties:
\begin{itemize}
	\item [I1]  Every rank-1 multilinear operator $\fhi\cdot y$, where $\fhi\in\Lbxx$ and $y\in Y$, is in the component $\Abxxy$ and $\Ab(\fhi\cdot y)\leq\|\fhi\|_\beta\;\|y\|$.
    \item [I2]  $\Lb(f_T) \leq\Ab(T)$ for all $T$ in  $\Abxxy$.
    \item [I3] If in the composition 
    \[\begin{array}{c}
\xymatrix{
\szzm\ar[r]^-\fr   &   \sxx\ar[r]^-{f_T}   &   Y\ar[r]^-S   &   W \\
}
\end{array}\]
    $\fr$ is a $\Sigma$-$\beta$-$\theta$-operator, $T$ is in $\Abxxy$ and $S:Y\into W$ is a bounded linear operator, then the multilinear operator $Sf_TR:\zpzm\into W$ belongs to $\Atzzmw$ and $\At(Sf_TR)\leq \|\rlin\| \Ab(T) \|S\|$.
\end{itemize}
\end{definition}

As the reader knows, there exist other approaches of ideals of multilinear operators highly influenced by ideas of Pietsch \cite{pietsch84}. This is the case of hyper-ideals developed by Botelho and Torres  \cite{botelho15} and multi-ideals developed by Floret and Hunfeld \cite{floret_hunfeld02}. This late is the most popular notion and  many of known examples of classes of multilinear operators fit in this framework. Next, we will see that a $\Sigma$-ideal $\ideal$ enjoys the benefits of multi-ideals although, strictly speaking, they are not comparable (since for defining a component in a $\Sigma$-ideal we have to provide a reasonable crossnorm $\beta$).

\begin{proposition}\label{multiideal}
Let $\ideal$ be a $\Sigma$-ideal in the class $\ban$. Let $(n, \xxx, Y,\beta)$ be an election in the class of Banach spaces. Then
\begin{itemize}
\item [(i)] Every rank-1 multilinear operator of the form
\begin{eqnarray*}
x_1^*\tens \dots \tens x_n^*\tens y,:\xpx &\into & Y\\
            \xxp  &\mapsto & x_1^*(x^1)\dots x_n^*(x^n)y,
\end{eqnarray*}
where $x_i^*\in X_i^*$, $1 \leq i \leq n$ and $y\in Y$, is in $\Abxxy$ and $\Ab(x_1^*\tens \dots \tens x_n^*\tens y)\leq \|x_1^*\|\dots \|x_n^*\|\| y\|$.

\item [(ii)] The inclusion $\Abxxy\subset \Lxxy$ is bounded and has norm at most 1.

\item [(iii)] $\Ab(\kk\times\dots\times\kk \ni (\lambda_1,\dots, \lambda_n)\mapsto \lambda_1\cdots\lambda_n\in\kk)=1$.

\item [(iv)] Let $T_i:Z_i\into X_i$, $1\leq i\leq n$ and $S:Y\into W$ be bounded linear operators and let $T$ in $\Abxxy$. Then $ST(T_1\times\dots\times T_n)$ is in $\mathcal{A}^\pi(\zzzn;W)$ and $A^\pi(STR)\leq \|T_1\|\dots\|T_n\|\Ab(T)\|S\|$.
\end{itemize}
\end{proposition}

\begin{proof}
The proof of (i) is immediate from Property I1; (ii) is a direct consequence of \eqref{lipnorm} and I2; (i) and (ii) imply (iii). For the proof of (iv) take $R=T_1\times\dots\times T_n$. The composition $STR$ can be factored as $Sf_T(\tens\circ (T_1\times\dots\times T_n))$, where $\tens:\xpx\into \sxx\subset\xtx$ is the natural multilinear map. It is clear that the composition $\tens\circ (T_1\times\dots\times T_n)$ is a $\Sigma$-$\beta$-$\pi$-operator. Then, I3 completes the proof.
\end{proof}

Among the benefits of the ideal property for $\Sigma$-ideals we have that it enable us to relate components of $n$-linear operators with components of $m$-linear operators not just $n$-linear operators with itself as is the case of multi-ideals.

An interesting situation occurs when considering those multilinear operators whose rank is contained in finite dimensional subspaces. Explicitly, let $\Fbxxy$ be the linear space of all bounded multilinear operators $T:\xpx\into Y$ such that their linearization $\tlin:\xtxb\into Y$ is bounded and  $T(\xpx)$ is contained in a finite dimensional subspace of $Y$. According to Proposition \ref{multiideal} (i), $\Fbxxy$ is contained in the component $\Abxxy$ for every $\Sigma$-ideal $\ideal$. Recall that this does not occur, in general, for multi-ideals (see \cite[Lemma 5.4]{jarchow07}).

The next step in our development is to specify the maximality for $\Sigma$-ideals. This property is intended to study multilinear operators by means of their behavior on finite dimensional spaces. For specifying maximality, fix a $\Sigma$-ideal $\ideal$ and an election $\xxyb$ in $\ban$ and let $T\in\Abxxy$. Let $E_i$ be a finite dimensional subspace of $X_i$ for $1\leq i \leq n$ and let $L$ be a cofinite dimensional subspace of $Y$. Consider the multilinear operator between finite dimensional normed spaces $Q_L T (I_{E_1}\times\dots\times I_{E_n}):\epe\into Y/L$ where $I_{E_i}$ is the natural inclusion of $E_i$ into $X_i$ and $Q_L$ is the canonical quotient map from $Y$ onto $Y/L$. In a maximal $\Sigma$-ideal we are able to calculate $\Ab(T)$ taking into account the values $A^\br(Q_L T (I_{E_1}\times\dots\times I_{E_n}))$ (see Remark \ref{betarestricted}). For shorten notation let $I_{\eee}:\epe\into \xtxb$ be the multilinear map defined by $I_{\eee}\xxp=\xxt$. Plainly, $I_{\eee}$ is a $\Sigma$-$\beta$-$\br$-operator.

\begin{definition}\label{maximality}
A $\Sigma$-ideal $\ideal$ is named maximal if
\[\Ab(T)=\sup  A^{\br}(Q_L f_T I_{\eee}:\epe\into Y/L)\]
for all $T$ in $\Abxxy$ where the suprema is taken over all $E_i\in\mathcal{F}(X_i)$, $1\leq i\leq n$, and $L\in\mathcal{CF}(Y)$.
\end{definition}

Most of the classical maximal ideals of linear operators admit a multilinear version following the approach of $\Sigma$-operators.

\begin{example}\label{boundedmultilinearoperators} \textbf{The $\mathbf{\Sigma}$-ideal of bounded multilinear operators.} Let $\sxxb$ be the metric space resulting by restricting the norm $\beta$ to $\sxx$. Recall that a multilinear operator $\Txxy$ is bounded if and only if $f_T:\sxxb\into Y$ is Lipschitz (see Section 2). Let $\mathcal{L}_{Lip}^\beta(\xxx;Y)$ be the Banach space $\Lxxy$ endowed with the norm $\Lb$. The $\Sigma$-ideal whose components are $\mathcal{L}_{Lip}^\beta(\xxx;Y)$ is a maximal $\Sigma$-ideal.
\end{example}

\begin{example}\textbf{The $\mathbf{\Sigma}$-ideal of Lipschitz $\mathbf{p}$-summing multilinear operators.} Following \cite{angulo18}, a multilinear operator $\Txxy$ is Lipschitz $\beta$-$p$-summing if there exists a positive constant $C$ such that
\[\sumim \|T(a_i)-T(b_i)\|^p\leq C^p \sup \left\{  \sumim |\fhi(a_i)-\fhi(b_i)|^p : \fhi\in B_{\Lbxx}  \right\}\]
for all finite sequences $(a_i)_{i=1}^m$ and $(b_i)_{i=1}^m$ in $\xpx$. The smallest constant $C$ as above is denoted by $\pi_p^{Lip,\beta}(T)$ and is named the Lipschitz $\beta$-$p$-summing norm of $T$. Let $\Pi_{p}^{Lip,\beta}(\xxx;Y)$ denote the Banach space of all Lipschitz $\beta$-$p$-summing multilinear operators from $\xpx$ to $Y$ endowed with the norm $\pi_p^{Lip,\beta}$. Then the collection defined as $\Pi_p^{Lip}=\bigcup \Pi_{p}^{Lip,\beta}(\xxx;Y),$ where the union is taken over all elections $(n,\xxx,Y,\beta)$ in $\ban$, is a maximal $\Sigma$-ideal. For the case $n=1$ see \cite[Chapter 2]{diestel95}.
\end{example}

\begin{example}\label{factorhilbert}\textbf{The $\mathbf{\Sigma}$-ideal of multilinear operators factoring through Hilbert spaces.} In \cite{fernandez-unzueta18a} it is specified the notion of these operators for the case $\pi$. An easy adaptation produced the next definition. Fix and election $(n,\xxx, Y, \beta)$ in $\ban$. The multilinear operator $T:\xpx\into Y$ factors through a Hilbert space with respect to $\beta$ if there exists a Hilbert space $H$, a subset $M$ of $H$, a bounded multilinear operator $A:\xpx\into H$ whose image is contained in $M$, and a Lipschitz function $B:M\into Y$ such that the next diagram commutes
\begin{small}
\begin{equation}\label{factorization}
\begin{array}{c}
\xymatrix{
\xpx\ar[dr]_{A}\ar[rr]^-T   &                              &Y\\
                                              &M\ar[ur]_B\ar[d]  &\\
                                              &H                           &
}
\end{array}.
\end{equation}
\end{small}Define $\Gamma^\beta(T)$ as $\inf \Lb(A)\, Lip(B)$ where the infimum is taken over all possible factorizations as above. Let $\Gamma^\beta(\xxx;Y)$ denote the Banach space of all these operators endowed with the norm $\Gamma^\beta$. Then the collection defined as $\Gamma=\bigcup \Gamma^{\beta}(\xxx;Y)$, where the union is taken over all elections $(n,\xxx,Y,\beta)$ in $\ban$ is a maximal $\Sigma$-ideal. For the case $n=1$ see \cite[Section 2.b]{pisier86}
\end{example}

\begin{example}\label{dominated}\textbf{The $\mathbf{\Sigma}$-ideal of $\mathbf{(p,q)}$-dominated multilinear operators.} In \cite[Definition 4.1]{fernandez-unzueta18b} we may find: Let $1\leq p,q \leq \infty$ such that $1/p+1/q \leq 1$. Take the unique $r\in[1,\infty]$ such that $1=1/r+1/p+1/q$. The  multilinear operator $T:\xpx\into Y$ is called $\beta$-$(p,q)$-dominated if there exists a constant $C>0$ such that
\[\|(\lev y_i^*  ,  T(x_i^1,\dots, x_i^n)- T(z_i^1,\dots, z_i^n)  \rev)_{i=1}^m\|_{r^*}\leq C\|(x_i^1\tens\cdots\tens x_i^n-z_i^1\tens\cdots\tens z_i^n)_{i=1}^m\|_{p}^{w,\beta}  \|(y_i^*)_{i=1}^m\|_{q}^w\]
holds for all finite sequences $(x_i^1,\dots, x_i^n)_{i=1}^m$, $(z_i^1,\dots, z_i^n)_{i=1}^m$ in $\xpx$ and $(y_i^*)_{i=1}^m$ in $Y$. Here
\[  \|(a_i-b_i )_{i=1}^m\|_{p}^{w,\beta} := \sup \left\{  \left(\sum_{i=1}^m |f_\fhi(a_i)-f_\fhi(b_i)|^{p}\right)^{1/p}  : \fhi\in B_{\Lbxx} \right\}. \]
Define $D_{p,q}^\beta(T)$ as the infimum of the constants $C$ as above. Let $\mathcal{D}_{p,q}^\beta(\xxx;Y)$ denote the Banach space of all $\beta$-$(p,q)$-dominated multilinear operators endowed with the norm $D_{p,q}^\beta$. Then the collection defined as $\mathcal{D}_{p,q}=\bigcup \mathcal{D}_{p,q}^{\beta}(\xxx;Y)$, where the union is taken over all elections $(n,\xxx,Y,\beta)$ in $\ban$ is a maximal $\Sigma$-ideal. See \cite[Section 19]{defant93} for the case $n=1$.
\end{example}

As we can see, in some cases the ideals has a simple demeanor. For instance, in the case of the $\Sigma$-ideal of bounded multilinear operators $[\mathcal{L};Lip]$ (see Example \ref{boundedmultilinearoperators}) all the components are isomorphic once we fix a positive integer $n$ and Banach spaces $\xxx$ and $Y$ (see \eqref{lipnorm}). Concretely, we have that
\[ \|Id: \mathcal{L}_{Lip}^\pi(\xxx; Y)\into  \mathcal{L}_{Lip}^\beta(\xxx; Y) \|\leq 1 \]
\[ \|Id: \mathcal{L}_{Lip}^\beta(\xxx; Y)\into  \mathcal{L}_{Lip}^\pi(\xxx; Y) \|\leq 2^{n-1}. \]
A similar result is valid for the $\Sigma$-ideal of multilinear operators which can be factored through a Hilbert space $[\Gamma,\gamma]$ (see Example \ref{factorhilbert}). In other situation it is not clear if the components are related in general, for instance we do not know if a Lipschitz $\pi$-$p$-summing multilinear operator is Lipschitz $\beta$-$p$-summing for some other reasonable crossnorm $\beta$. In Section 5.2.1 we explore more about this phenomena.


\section{$\Sigma$-Tensor Norms}

Throughout this section we extend the concept of tensor norm for two factors (see \cite[Section 12.1]{defant93}) to the case of several factors having in mind that we are interested in studying multilinear operators by means of their associated $\Sigma$-operators. The outcome of this approach is the concept of $\Sigma$-tensor norm. As we will see, these norms appear in two different version: \emph{on spaces} and \emph{on duals}.

\subsection{$\mathbf{\Sigma}$-Tensor Norms on Spaces}

This type of tensor norms is a procedure that assign a norm to each tensor product of the form $\xtxty$ such that their topological dual identifies a component of a $\Sigma$-ideal, that is, multilinear operators from $\xpx$ to $Y^*$ (see Theorem \ref{rt} and \eqref{conclusion}).

\begin{definition}\label{tns}
A {\bf $\Sigma$-tensor norm on spaces} $\alpha$ on the class of $\ban$ assigns, to each election $\xxyb$ in $\ban$, a norm $\alpha^\beta$ on the algebraic tensor product $\xtxty$ such that:
\begin{itemize}
    \item [S1]  $\alpha^\beta((p-q)\tens y)\leq\beta(p-q)\,\|y\|$ for every $p,q\in\sxx$ and $y\in Y$.
    \item [S2]  For every $\fhi\in\Lbxx$ and $y^*\in Y$ the linear functional
    \begin{eqnarray*}
     \fhi\tens y^*:\xtxty &\into &		   \kk\\
	                \xxyt         &\mapsto &  f_\fhi(\xxt)y^*(y)
    \end{eqnarray*}
    is bounded and $\|\fhi\tens y^*\|\leq\|\fhi\|_\beta\,\|y^*\|$.
    \item [S3]  If $\fr:\szzm\into\sxx$ and $S:W\into Y$ denote a $\Sigma$-$\beta$-$\theta$-operator and a bounded linear operator respectively, then the tensor product operator
    \begin{eqnarray*}
    \fr\tens S:\left(\ztzm\tens W,\alpha^\theta\right)   &\into &		  \left(\xtxty,\alpha^\beta\right)\\
	       z^1\tens\dots\tens z^m\tens w             &\mapsto &  \fr(z^1\tens\dots\tens z^m)\tens S(w)
    \end{eqnarray*}
is bounded and $\|f_R\tens S\|\leq\|\rlin\|\,\|S\|$.
\end{itemize}
\end{definition}

The authors of \cite{floret_hunfeld02} also develop the notion of tensor norms for tensor products of several factors. Next, we show that $\Sigma$-tensor norms on spaces enjoy the same benefits.

\begin{proposition}
Let $\alpha$ be a $\Sigma$-tensor norm on spaces in the class $\ban$. Let $(n,\xxx,Y,\beta)$ be an election in the class of Banach spaces. Then
\begin{itemize}
\item [(i)] $\ab(\xxt\tens y)\leq \|x^1\|\dots \|x^n\|\|y\|$ for all $x^i\in X_i$ and $y\in Y$.

\item [(ii)] The functionals $x_1^*\tens\dots\tens x_n^*\tens y^*:(\xtxty,\ab)\into \kk$ are bounded and $\|x_1^*\tens\dots\tens x_n^*\tens y^* \|\leq \|x_1^*\|\dots \|x_n^*\| \|y^*\|$ for all $x_i^*\in X_i^*$ and $y\in Y^*$.

\item [(iii)] If $T_i:Z_i\into X_i$, $1\leq i\leq n$ and $S:Y\into W$ be bounded linear operators, then
\begin{eqnarray*}
T_1\tens\dots\tens T_n\tens S:\left(\ztz\tens W,\alpha^\pi\right)   &\into &		  \left(\xtxty,\alpha^\beta\right)\\
	       z^1\tens\dots\tens z^n\tens w             &\mapsto &  T_1(z^1)\tens\dots\tens T_n(z^n)\tens S(w)
\end{eqnarray*}
is bounded and $\|T_1\tens\dots\tens T_n\tens S\|\leq \|T_1\|\dots\|T_n\|\|S\|$.
\end{itemize}
\end{proposition}

\begin{proof}
Items (i) and (ii) are immediate from definition since $\beta$ is a reasonable crossnorm. Item (iii) follows from the fact that the composition $R=\tens\circ (T_1\times\dots\times T_n)$ is a $\Sigma$-$\beta$-$\pi$-operator, where $\tens:\xpx\into \sxx\subset\xtxb$ is the natural multilinear map. The proof is complete by applying Property S3.
\end{proof}

The next step in our development is to specify a proper notion of finitely generation for $\Sigma$-tensor norms on spaces (see \cite[Section 12.4]{defant93}). Roughly speaking, a finitely generated $\Sigma$-tensor norm on spaces $\alpha$ allows to calculate $\ab(u)$ of any $u$ in $\xtxty$ by considering those finite dimensional tensor products $\etetf$, where $E_i\subset X_i$ and $F\subset Y$, which contain $u$. As well as maximality for $\Sigma$-ideals we have to consider the reasonable crossnorm $\br$ and take into account the values $\alpha^\br(u;\eee, F)$ (see Remark \ref{betarestricted}).

\begin{definition}\label{finitelygenerated}
A $\Sigma$-tensor norm on spaces $\alpha$ is named finitely generated if
\[\ab(u;\xxx,Y):=\inf	\abr(u;\eee,F)	\]
where the infimum is taken over all $E_i\in\mathcal{F}(X_i)$, $1\leq i\leq n$ and $F\in\mathcal{F}(Y)$ such that $u$ is contained in $\etetf$.
\end{definition}

\subsubsection*{The Greatest and Least $\Sigma$-Tensor Norms on Spaces}

Examining properties $S1$ and $S_2$ of Definition \ref{tns} and motivated by the injective and projective tensor norms for two factors, we define for each election $\xxyb$ in the class $\ban$
\begin{align*}
\pi^\beta(u) &:=\inf \left\{  \sumim \beta(\pimqi)\|y_i\|  \;|\;  u=\sumim (\pimqi)\tens y_i \; ; \; p_i,\,q_i\in\sxx, y_i\in Y \right\}\\
\varepsilon^\beta(u)&:=\sup \left\{\;  |\lev \fhi\tens y^* , u \rev|  \;|\;  \|\fhi\|_\beta\leq1,\; \|y^*\|\leq 1 \;\right\}.
\end{align*}
for all $u\in \xtxty$.

Let $\pi$ denote the assignment defined as follows: For each election $(n,\xxx, Y, \beta)$ in the class $\ban$, $\pi$ assigns the norm $\pi^\beta$. Similarly define $\varepsilon$. Direct from definition we have that $\pi^\beta$ and $\varepsilon^\beta$ give rise to $\Sigma$-tensor norms on spaces. Moreover, in the next proposition we prove that $\pi$ and $\varepsilon$ are the greatest and least $\Sigma$-tensor norms on spaces, respectively.

\begin{proposition}
Let $\xxyb$ be an election in the class $\ban$. A norm $\ab$ on $\xtxty$ verifies S1 and S2 if and only if
\[\varepsilon^\beta(u)\leq \ab(u)\leq\pi^\beta(u)\qquad \forall\; u\in\xtxty.\]
\end{proposition}

The proof of this proposition just requires elementary tensor products techniques and are easily adapted from the case of two factors, so we omit it.

\begin{example} Let $\xxyb$ be an election in $\ban$. Define
\[\gamma^\beta(u)=\inf  \|(a_i-b_i)\|_2^\beta\,\|(y_i)\|_2\]
where the infimum is taken over all representation of the form $u=\sum_{i=1}^m(\pimqi)\tens y_i$, such that $p_i,q_i\in\sxx, y_i\in Y$ and $(p_i,q_i)\leq_\beta (a_i,b_i)$. Here, $\|(a_i-b_i)\|_2^\beta=\left(\sum_{i=1}^m \beta(p_i-q_i)^2\right)^{1/2}$ and $(p_i,q_i)\leq_\beta (a_i,b_i)$ means that $\sum_{i=1}^m |f_\fhi(p_i)-f_\fhi(q_i)|^2\leq \sum_{i=1}^m |f_\fhi(a_i)-f_\fhi(b_i)|^2 $ for all $\fhi$ in $\Lbxx$. The case $n=1$ can be consulted in \cite[Section 2.b]{pisier86}
\end{example}

\begin{example}{\bf The Laprest\'e $\Sigma$-Tensor Norms on Spaces.} We recall the next definition from \cite[Definition 4.1]{fernandez-unzueta18a}; the case $n=1$ is detailed in \cite[Section 12.5]{defant93}. Let $1\leq p,q\leq \infty$ such that $1/p+1/q\geq1$. Take the unique $r\in[1,\infty]$ with the property $1=1/r+1/q^*+1/p^*$. The Laprest\'e $\Sigma$-tensor norm on spaces $\alpha_{p,q}^\beta$ on $\xtxty$ is defined by
\[\alpha_{p,q}^\beta(u):=\inf  \|(\lambda_i)_{i=1}^m\|_r \|(\pimqi)_{i=1}^m\|_{q^*}^{w,\beta} \|(y_i)_{i=1}^m\|_{p^*}^w\]
where the infimum is taken over all representations of the form $u=\sum_{i=1}^m \lambda_i(p_i-q_i)\tens y_i$ with $\lambda_i\in \kk$, $p_i,q_i\in\sxx$, $y_i\in Y$ (see Example \ref{dominated}).

As it is well known, the Chevet-Saphar tensor norms are particular cases of the Laprest\'e tensor norm. The corresponding cases acquire the form
\begin{align*}
 d_p^\beta(u):=\alpha_{1,p}^\beta(u)&=\inf\left\{  \|(\pimqi)_i \|_{p}^{w,\beta}  \|(y_i)_i\|_{p^*}   \,\Big|\,  u=\sumim (\pimqi)\tens y_i  \right\}\\
 w_p^\beta(u):=\alpha_{p,p^*}^\beta(u)&=\inf\left\{  \|(\pimqi)_i \|_{p}^{w,\beta}  \|(y_i)_i\|_{p^*}^w   \,\Big|\,  u=\sumim (\pimqi)\tens y_i  \right\}\\
 g_p^\beta(u):=\alpha_{p,1}^\beta(u)&=\inf\left\{  \|(\pimqi)_i \|_{p}^\beta  \|(y_i)_i\|_{p^*}^w   \,\Big|\,  u=\sumim (\pimqi)\tens y_i  \right\}\\
 \alpha_{1,1}^\beta(u)=d_{1}^\beta(u)&=g_1^\beta(u)=\pi^\beta(u)
\end{align*}
\end{example}

The assignments $\pi$, $\varepsilon$, $\gamma_2$ and $\alpha_{p,q}$ are examples of finitely generated $\Sigma$-tensor norms on spaces. The cases $\pi$ and $\varepsilon$ follows from definition, the case $\gamma_2$ is \cite[Proposition 4.3]{fernandez-unzueta18b} and the Laprest\'e case is \cite[Proposition 4.3]{fernandez-unzueta18a}.

\subsection{$\mathbf{\Sigma}$-Tensor Norms on Duals}

In ideals of operators, not necessarily linear or multilinear, operators which have range contained in finite dimensional subspaces are of fundamental importance. As we saw in Section 3, $\Fbxxy\subset \Abxxy$ holds linearly for every $\Sigma$-ideal $\ideal$. It is not difficult to see that $\Fbxxy$ is linearly isomorphic to $\Lbxx\tens Y$. In this section we give a precise definition of those norms on $\Lbxx\tens Y$ which identify the space $\Fbxxy$ endowed with the norm $\Ab$ (see Theorem \ref{tndsideal} and \eqref{conclusion}).

The most technical requirement of the definition of $\Sigma$-tensor norms on duals is the uniform property. In Section 4.2 we saw that the uniform property for $\Sigma$-tensor norms on spaces requires the notion of the so called $\Sigma$-$\beta$-$\theta$-operators. Vaguely, for defining $\Sigma$-tensor norms on duals we need a dual concept of $\Sigma$-$\beta$-$\theta$-operators.

Let $n, m$ positive integers, $\xxx$, $\zzzm$ be Banach spaces and $\beta$ and $\theta$ be two reasonable crossnorms on $\xtx$ and $\ztzm$, respectively. If the linear operator $A:\Lbxx\into\Ltzzm$ is such that its adjoint linear operator $A^*:\ztzmt^{**}\into \xtxb^{**}$ verifies
\[A^*K_{\ztzt}(\szzm)\subset K_{\xtxb}(\sxx),\]
we say that $A$ {\bf preserves $\mathbf{\Sigma}$}.

\begin{definition}\label{tnd}
A {\bf $\mathbf{\Sigma}$-tensor norm on duals} $\nu$ on the class of Banach spaces assigns, to each election $\xxyb$ in $\ban$, a norm $\vb$ on the tensor product $\Lbxx\tens Y$ with the following properties:
\begin{itemize}
	\item  [D1] $\vb(\fhi\tens y)\leq\|\fhi\|_\beta \|y\|$ for every $\fhi\in\Lbxx,\; y\in Y.$
    \item  [D2] For every $p,q\in\sxx$ and $y^*\in Y^*$ the linear functional
    \begin{eqnarray*}
     (\pmq)\tens y^*:\Lbxx\tens Y &\into &		   \kk\\
	                \fhi\tens y &\mapsto & (f_\fhi(p)-f_\fhi(q))y^*(y)
    \end{eqnarray*}
    is bounded and $\|(\pmq)\tens y^*\|\leq\beta(\pmq)\|y^*\|$.
    \item  [D3] If $A:\Lbxx\into\Ltzzm$ and $B:Y\into W$ denote bounded linear operators where $A$ preserves $\Sigma$, then the linear operator
    \begin{eqnarray*}
    A\tens B:\left(\Lbxx\tens Y,\vb\right)     & \into &\left(\Ltzzm\tens W,\vt\right)\\
    \fhi\tens y& \mapsto &  A(\fhi)\tens B(y)
    \end{eqnarray*}
    is bounded and $\|A\tens B\|\leq\|A\|\,\|B\|$.
    \end{itemize}
\end{definition}

The next step int our development is to define a well behavior of $\Sigma$-tensor norm on duals in the general case of Banach spaces regarding finite dimensional spaces. To be explicit we define when a $\Sigma$-tensor norm on duals is cofinitely generated.

Consider Banach spaces $\xxx$ and finite dimensional subspaces $E_i$ of $X_i$ for $1 \leq i \leq n$. Let $R_{\eee}:\Lbxx\into \mathcal{L}^{\br}(\eee)$ be the linear operator which maps each $\fhi$ to its restriction to $\epe$ (see Remark \ref{betarestricted}). The analogous case of this property for the case $n=1$ is defined in \cite[Section 12.4]{defant93}.

\begin{definition}\label{cofinitelygenerated}
A $\stn$ on duals $\nu$ on the class $\ban$ is named cofinitely generated if
\[\vb(v;\Lbxx,Y):=\sup \vbr\left( R_{\eee}\tens Q_L (v); \Lbree, Y/L\right)\]
where the suprema is taken over all $E_i\in \mathcal{F}(X_i)$ and $L\in\mathcal{CF}(Y)$.
\end{definition}

\subsubsection*{The Greatest and the Least $\Sigma$-Tensor Norms on Duals}

For each election $\xxyb$ in $\ban$ define
\begin{align*}
\pi_\beta(v)&:= \inf \left\{\;\sumim \|\fhi\|_\beta \|y_i\| \;\Big|\; v=\sumim\fhi_i\tens y_i \;\right\}\\
\varepsilon_\beta(v)&:=\sup\left\{\;	  |\lev  (\pmq)\tens y^*  ,  v \rev|\;\Big|\;  \beta(\pmq)\leq 1,\, \|y^*\|\leq 1  \;\right\}
\end{align*}
for all $v\in\Lbxx\tens Y$. As it is expected, these norms are the greatest and least $\Sigma$-tensor norms on duals, respectively.

\begin{proposition}
Let $\xxyb$ be an election in the class $\ban$. A norm $\vb$ on the tensor product $\Lbxx\tens Y$ verifies D1 and D2 if and only if
\[\varepsilon_\beta(v)  \leq\vb (v)  \leq\pi_\beta(v)\quad \forall v\in \Lbxx\tens Y.\]
\end{proposition}

The proof of this proposition is easily adapted from the case of two factors, so we omit it.

\subsection{The Duality Relation between $\Sigma$-Tensor Norms in Finite Dimensions}

In this section we exhibit the relation between the two types of $\Sigma$-tensor norms in the class of finite dimensional normed spaces. As we will see, this relation is in a dual fashion.

In general, the linear spaces $\Lbxx\tens Y$ and $\xtxtyd$ have not the same dimension. For this reason, we do not expect to relate the two types of tensor norms in general Banach spaces in an  isomorphically fashion (see Theorem \ref{dt}). Nevertheless, the finite dimensional case is simpler and more illustrative.

\begin{theorem}\label{tnstotnd}
Every $\Sigma$-tensor norm on spaces $\alpha$ on the class $\finn$ defines a $\Sigma$-tensor norm on duals $\nu$  on the class $\finn$ by
\begin{equation}\label{linearduality}
 \left(\Lbee\tens F,\vb\right):=\left(\etetfd,\ab\right)^*.
\end{equation}
Every $\Sigma$-tensor norm on duals $\nu$ on the class $\finn$ defines a $\Sigma$-tensor norm on spaces $\alpha$ on the class $\finn$ by
\[\left(\etetf,\ab\right):=\left(\Lbee\tens F^*,\vb\right)^*\]
\end{theorem}

Before proving this theorem let us make an observation: If we restrict \eqref{linearduality} to the case $n=1$ we obtain $(E^*\tens F,v)=(E\tens F^*,\alpha)^*$, or equivalently, $(E\tens F,v)=(E^*\tens F^*,\alpha)^*$ for all finite dimensional normed spaces. This last equation is nothing more than the construction of the dual tensor norm $\nu=\alpha'$ of $\alpha$ for the case of two factors. Under this viewpoint, we may consider \eqref{linearduality} as an extension of the procedure to define the dual tensor norm of a given tensor norm. In our setting we have that the dual norm of a $\Sigma$-tensor norm on spaces is a $\Sigma$-tensor norm on duals (and reciprocally).

\begin{proof}
We only proof that \eqref{linearduality} gives rise to a $\Sigma$-tensor norm on duals since the opposite direction is analogous. Let $\alpha$ be a $\Sigma$-tensor norm on spaces and let $\eefb$ be an election in $\finn$.

{\bf $\vb$ verifies D1}: S2 is equivalent to D1 since for every $\fhi\in\Lbee$ and $y\in F$ we have $\vb(\fhi\tens y):= \|\fhi\tens y  :  \left(\etetfd,\ab\right) \into \kk \|\leq \|\fhi\|_\beta\, \|y\|$.

{\bf$\vb$ verifies D2}: Since the involved spaces are finite dimensional we have that the functionals defined by $(\pmq)\tens y^*$ are bounded. The definition of $\vb$ implies
\[\sup\limits_{\vb(v)\leq 1}|\lev (\pmq)\tens y^* \,,\, v \rev|=\ab((\pmq)\tens y^*)\leq\beta(\pmq)\ \|y^*\|.\]

{\bf$\vb$ verifies D3}: Let $(m,M_1, \dots, M_m, N, \theta)$ be an election in the class $\finn$. Consider a bounded linear operator that preserves $\Sigma$ $A:\Lbee\into\mathcal{L}^\theta(M_1, \dots, M_m)$ and a bounded linear operator $B:N\into F$. The finite dimensional assumption lets us consider $A^*:(M_1\tens \cdots\tens M_m,\theta)\into (E_1\tens \cdots\tens E_n,\beta)$. The linearity of $A^*$ lets us define the multilinear operator
\begin{eqnarray*}
T:M_1\times\dots\times M_m &\into &   (E_1\tens \cdots\tens E_n,\beta)\\
                               (z_1,\dots, z_m) &\mapsto & A^*(z_1\tens\cdots\tens z_m).
\end{eqnarray*}
The universal property of the tensor product implies $\tlin=A^*$. Hence, $\tlin:(M_1\tens\cdots\tens M_m,\theta)\into (\ete,\beta)$ is bounded and $\|\tlin\|=\|A\|$. Even more, the set $\{ A^*(p) \,|\,  p\in\Sigma_{M_1,\dots, M_m} \}$ is contained in $\Sigma_{\eee}$ since $A$ preserves $\Sigma$. In other words, the $\Sigma$-operator $\ft:\Sigma_{M_1,\dots, M_m}^\theta \into (E_1\tens \cdots\tens E_n,\beta)$ is a $\Sigma$-$\beta$-$\theta$-operator. The uniform property of $\alpha$ implies that $\ft\tens B^* :\left(M_1,\tens\cdots\tens, M_m\tens N,\alpha^\theta\right) \into   \left(\etetfd,\ab\right)$ is bounded and $\|\ft\tens B^*\|\leq\|A\|\,  \|B^*\|$.

The boundedness of $A\tens B:\left(\Lbee\tens F,\vb\right)\into \left(\mathcal{L}^\theta(M_1,\dots, M_n)\tens N,\vt\right)$ is deduced from
\begin{align*}
|\lev A\tens B (v)\,,\, u\rev\big| =\Big|\lev v\,,\, \ft\tens B^*(u)\rev\big|&\leq \vb(v)\,\ab\left(\ft\tens B^*(u) ; \eee, F^*\right)\\
            &\leq \vb(v)\,\|A\|\,\|B\|\,\alpha^\theta(u  ; \mmm, N^*).
\end{align*}
\end{proof}

\begin{example}
Adapting the case of two factors leads us to the fact that $\varepsilon_\beta$ is dual to $\pi^\beta$. In other words they verify a relation as in \eqref{linearduality}. In this particular case we can say more: the space $ (\Lbxx\tens Y ,\varepsilon_{\beta})$ is isometrically contained in the dual space $(\xtxtyd,\pi^\beta)^*$ via the morphism given by evaluation for all election $(n,\xxx,Y,\beta)$ in $\ban$ (see Sections 4.1.1, 4.2.1, Theorem \ref{dt} and \eqref{conclusion}).
\end{example}

\subsection{$\mathbf{\Sigma}$-Tensor Norms and $\mathbf{\Sigma}$-Ideals in the Finite Dimensions}

Next, we show that each $\Sigma$-tensor norm (on duals) gives rise to a unique $\Sigma$-ideal in finite dimensions. For the linear case of this relation the reader may check \cite[Sections 17.1 and 17.2]{defant93}

\begin{theorem}\label{tndsideal}
Every $\Sigma$-tensor norm on duals $\nu$ on $\finn$ defines a $\Sigma$-ideal $\ideal$ on $\finn$ by
\[\Abeef:=\left(\Lbee\tens F,\vb\right).\]
Conversely, every $\Sigma$-ideal $\ideal$ on $\finn$ defines a $\stn$ on duals $\nu$ on $\finn$, by
\[\left(\Lbee\tens F,\vb\right):=\Abeef.\]
\end{theorem}

\begin{proof}
We only prove that a $\Sigma$-tensor norm on duals defines a $\Sigma$-ideal since the converse is completely analogous. Let $\nu$ be a $\stn$ on duals. Plainly, the assignment
$\fhi\tens y\mapsto \fhi\cdot y$ defines a linear isomorphism between $\Lbee\tens F$ and $\Abeef$.

\textbf{$\ideal$ verifies I1}: Immediate from D1.

\textbf{$\ideal$ verifies I2}: Let $T=\sum_{i=1}^l\fhi_i\cdot y_i$ in $\Abeef$ and take $p,q\in\Sigma_{\eee}$ and $y^*\in Y^*$. Algebraic manipulations lead to $\lev  y^*  ,  f_T(p)-f_T(q) \rev = \lev  (\pmq)\tens y^*  \,,\,  \sum_{i=1}^l\fhi_i\tens y_i  \rev$. Hence, Property D2 of $\nu$ and $\vb(v_T)=\Ab(T)$ assert that
\[|\lev  y^*  ,  f_T(p)-f_T(q) \rev|\leq \beta(\pmq)\,  \|y^*\|\,  \Ab(T).\]


\textbf{$\ideal$ verifies I3}: Consider the composition
\[\begin{array}{c}
\xymatrix{
\Sigma_{M_1\dots M_m}\ar[r]^-\fr   &   \Sigma_{\eee} \ar[r]^-{f_T}   &   F\ar[r]^-S   &   N \\
}
\end{array},\]
where $\fr$ is a $\Sigma$-$\beta$-$\theta$-operator, $T$ is in $\Abeef$ and $S:F\into N$ is a bounded linear operator. According to the isomorphism between $\mathcal{A}^\theta(M_1\dots M_m;N)$ and $\mathcal{L}^\theta(M_1\dots M_m)\tens N$ we have that the multilinear operator $Sf_T R$ is associated with $(\rlin^*\tens S) (v_T)$. The linear operator $\rlin^*: (\ete,\beta)^* \into (M_1\tens\cdots\tens M_m,\theta)^*$ preserves $\Sigma$ since $f_R$ is a $\Sigma$-$\beta$-$\theta$-operator. Hence, applying D3 we obtain $\At(S f_T R)  =     \vt\left(  \rlin^*\tens S (v_f) \right) \leq  \|\rlin\|\,  \|S\|\, \vb(v_f) = \|\rlin\|\,  \|S\|\,  \Ab(T)$.
\end{proof}


\section{Duality and Representation Theorems}

In sections 4.3 and 4.4 we established the relation between $\Sigma$-tensor norms and $\Sigma$-ideals in the class $\finn$. Next we extend these relations for Banach spaces. The outcomes are the duality theorem for $\Sigma$-tensor norms and the representation theorem for maximal $\Sigma$-ideals.


\subsection{The Duality Theorem for $\Sigma$-Tensor Norms}

As we saw in Theorem \ref{tnstotnd} the two types of $\Sigma$-tensor norms are in duality relation on the class of finite dimensional normed spaces. For extending this relation to general Banach spaces we need to extend a given $\Sigma$-tensor norms on spaces on finite dimensional normed spaces to Banach spaces.

Let $\alpha$ be a $\Sigma$-tensor norm on spaces defined on the class $\finn$. Let $(n,\xxx,Y)$ be an election in $\ban$. For every $u$ in $\xtxty$ define
\[\ab(u;\xxx,Y):=\inf	\abr(u;\eee,F)	\]
where the infimum is taken over all $E_i\in\mathcal{F}(X_i)$, $1\leq i\leq n$ and $F\in\mathcal{F}(Y)$ such that $u$ is contained in $\etetf$. Following the definition and using elementary tensor product techniques it is possible to prove that $\alpha$ is a finitely generated $\Sigma$-tensor norm on spaces on the class $\ban$.

According to previous definition and Theorem \ref{tnstotnd}, every $\Sigma$-tensor norm on duals $\nu$ in the class $\ban$ gives rise to a unique finitely generated $\Sigma$-tensor norm on spaces $\alpha$ on the class $\ban$ such that $\left(\etetf,\ab\right)$ is isometrically isomorphic to the dual space $\left(\Lbee\tens F^*,\vb\right)^*$ for all election $(n,\eee, F)$ in $\finn$.

\begin{theorem}\label{dt}{\bf (Duality Theorem of $\Sigma$-Tensor Norms)}
Let $\nu$ be a cofinitely generated $\Sigma$-tensor norm on duals and let $\alpha$  be the finitely generated $\Sigma$-tensor norm on spaces  defined by $\nu$. Then $\left(\Lbxx\tens Y,\vb\right)$ is isometrically contained in $\left(\xtxtyd,\ab\right)^*$ for all elections $\xxyb$ in $\ban$.
\end{theorem}

\begin{proof}
Let $\xxyb$ be an election in $\ban$. Let $v=\sum_{j=1}^l \fhi_j\tens y_j$ in $\Lbxx\tens Y$. If $\ab(u; \xxx, Y)<1$ then there exist  $E_i\in\mathcal{F}(X_i)$, $1\leq i \leq n$ and $F\in\mathcal{F}(Y^*)$ such that $u\in\etetf$ and $\abr(u; \eee, F)<1$. The space $L=\{y  \,|\, y^*(y)=0 \mbox{ for all } y^*\in F\}$ is an element of $\mathcal{CF}(Y)$ such that $F$ is isometric to $(Y/L)^*$. Take any representation of $u$ in $\etetf$ of the form $\sum_{i=1}^m(\pimqi)\tens y_i^*$. Then, algebraic manipulations show that $\lev  v  ,  u  \rev = \lev  R_{\eee}\tens Q_L(v)  ,  \sum_{i=1}^m(\pimqi)\tens z_i^* \rev$ for some $z_i^*\in (Y/L)^*$, $1\leq i\leq m$. By assumption, the space $(\Lbree\tens Y/L,\vbr)$ is isometric to $(  \ete\tens (Y/L)^*,\abr)^*$. Hence, $|\lev  v  ,  u  \rev|\leq \vbr(R_{\eee}\tens Q_L(v)) \, \abr\left(\sum_{i=1}^m (\pimqi)\tens z_i^*\right).$ Since $\abr\left(\sum_{i=1}^m (\pimqi)\tens z_i^*\right)<1$, we have that $|\lev  v  ,  u  \rev|\leq \vb\left(v;\,  \Lbxx, Y\right)$. After taking suprema over all $\ab(u)<1$ in last inequality  we obtain
\[\|v:\left(\xtxtyd,\ab\right)\into \kk\|\leq\vb\left(v;\,  \Lbxx Y\right).\]


For the converse inequality, let $E_i\in\mathcal{F}(X_i)$, $1 \leq i \leq n$, $L\in\mathcal{CF}(Y)$ and $\eta>0$. There exist $u$ in the space $ E_1\tens\cdots E_n\tens (Y/L)^*$ such that $\abr(u;  \eee, (Y/L)^*)<1$ and $\vbr(R_{\eee}\tens Q_L(v),\Lbree, Y/L)\, (1-\eta)\leq |\lev  R_{\eee}\tens Q_L(v)  \,,\,  u\rev|$. The space $(Y/L)^*$ is isometrically isomorphic to a finite dimensional subspace of $F$ of $Y^*$. Then, after algebraic manipulations, we have
\[|\lev  R_{\eee}\tens Q_L(v)  \,,\,  u\rev| = |\lev  v  ,  u  \rev| \leq \|v:\left(\xtxtyd,\ab\right)\into \kk\|.\]
After taking suprema over all $E_i$, $1 \leq i \leq n$ and $L$ as above we obtain that
\[\vb(u;\  \Lbxx, Y)(1-\eta)\leq\|v:\left(\xtxtyd,\ab\right)\into \kk\|\]
holds for all $\eta>0$.
\end{proof}

The linear version of this result can be consulted in \cite[Section 15.5]{defant93}.

\subsection{The Representation Theorem for Maximal $\Sigma$-ideals of Multilinear Operators}

For proving this result, we need two preliminary results about maximal $\Sigma$-ideals and finitely generated $\Sigma$-tensor norms. A detailed exposition of this result for the case of linear operators can be found in \cite[Section 17]{defant93}.

\begin{proposition}\label{regular}
Let $\ideal$ be a maximal $\Sigma$-ideal. Then $T\in\Abxxy$ if and only if $K_Y T\in \Abxxydd$.
\end{proposition}

\begin{proof}
The ideal property asserts that $K_Y T\in\Abxxydd$ and $\Ab(K_Y T)\leq \Ab(T)$ whenever $T\in\Abxxy$.

Conversely, suppose $K_Y T\in\Abxxydd$. Let $E_i\in\mathcal{F}(X_i)$, $1 \leq i \leq n$ $L\in\mathcal{CF}(Y)$ and $\eta>0$. There exist a finite dimensional subspace $F$ of $Y^*$ isometric to $(Y/L)^*$ via $Q_L^*$. Consider the subspace $H$ of $Y^{**}$ defined as the span of the set $K_Y f_T(\Sigma_{\eee})$. By the Principle of Local Reflexivity there exist a finite dimensional subspace $G$ of $Y$ and an isomorphism $\psi:H\into G$ with $\|\psi\|\leq 1+\varepsilon$ such that $ \lev y^* , \psi(y^{**})  \rev =   \lev y^{**} , y^*  \rev$ for all $y^*\in F$ and $y^{**}\in H$. Then, algebraic manipulations lead to $\lev \varphi  \,,\, Q_L f_T I_{\eee}(p)  \rev = \lev \varphi  \,,\, Q_L\psi K_Y f_T(p)  \rev$ for all $\varphi \in (Y/L)^*$ and $p\in\Sigma_{\eee}$. In other words $Q_L f_T I_{\eee}=Q_L\psi K_Y T$. Finally, the ideal property implies
\[\Ab(Q_L f_T I_{\eee})=\Ab(Q_L\psi K_Y T)\leq\left(1+\eta \right)\,  A^\beta(K_YT)\]
which ensures $\Ab(Q_L f_T I_{\eee})\leq \Ab(K_YT)$. The proof is complete after taking suprema, in previous inequality, over all $E_i$, $1\leq i \leq n$ and $L$ as above.
\end{proof}

As usual, every bounded functional $\fhi:X_1\tens_\pi\dots\tens_\pi X_n\tens_\pi Y\into\kk$ defines a bounded multilinear operator $T_\fhi:\xpx\into Y^*$ given by $\lev T_\fhi\xxp , y\rev:= \fhi (\xxt\tens y)$ for all $x^i\in X_i$, $1\leq i\leq n$ and $y\in Y$. The canonical extension of $\fhi$ is denoted by $\varfhi$ and defined by
\begin{eqnarray*}
\overline{\fhi}:\xtx\tens Y^{**}                       &\into      &   \kk\\
x^1\tens\dots\tens x^n\tens y^{**}   &\mapsto &  \lev  y^{**} , T_\fhi(\xxt)  \rev.
\end{eqnarray*}
If we consider $\xtxty$ as a linear subspace of $\xtx\tens Y^{**}$, the functional $\overline{\fhi}$ is actually an extension of $\fhi$ since for every $x^i\in X_i$, $1\leq i\leq n$ and $y\in Y$ we have $\overline{\fhi}(\xxt\tens y)=\lev  T_\fhi(\xxt)  ,  y  \rev =\fhi(\xxt\tens y)$.

On the other hand, every bounded multilinear operator $T:\xpx\into Y$ defines a bounded functional $\fhi_T:X_1\tens_\pi\dots\tens_\pi X_n\tens_\pi Y^*\into\kk$ given by $\fhi_T(\xxt\tens y^*)=\lev y^*,  T\xxp \rev$ for all $x^i\in X_i$, $1\leq i\leq n$ and $y^*\in Y^*$. It is not difficult to see that $\fhi_{(T_\fhi)}=\varfhi$ for all bounded functional $\fhi:X_1\tens_\pi\dots\tens_\pi X_n\tens_\pi Y\into\kk$ and $T_{(\fhi_T)}=K_YT$ for all bounded multilinear operator $\Txxy$.

\begin{proposition}\label{aronbernerextension}
Let $\alpha$ be a finitely generated $\Sigma$-tensor norm on spaces on the class $\ban$ and let $\xxyb$ be an election in $\ban$. Then:
\begin{itemize}
\item [(i)]  $\left(\xtxty,\ab\right)$ is a closed subspace of $\left(\xtxtydd,\ab\right)$.
\item [(ii)] $\fhi\in\left(\xtxty,\ab\right)^*$ if and only if $\overline{\fhi}\in\left(\xtxtydd,\ab\right)^*$.
In this situation $\|\fhi\|=\|\overline{\fhi}\|$.
\end{itemize}
\end{proposition}

\begin{proof}
(i): In this proof we have identified the spaces $Y$ and $K_Y(Y)$. By uniformity of $\alpha$ it is clear that $\ab(u; \xxx ,Y^{**})\leq\ab(u; \xxx, Y)$. For the converse inequality let $u\in\xtxty$, $\eta>0$ and fix a representation of $u$ of the form $\sum_{i=1}^m x_i^1\tens\dots\tens x_i^n\tens y_i$. There exist finite dimensional subspaces $E_i\subset X_i$ and $F\subset Y^{**}$ such that $u\in\etetf$ and
\[\abr(u;  \eee, F)\leq (1+\eta)\,  \ab(u ; \xxx,Y^{**}).\]
We may assume that $y_i\in F$, $1\leq i \leq n$. The Principle of Local Reflexivity ensures the existence of a finite dimensional subspace $G$ of $Y$ and an isomorphism $\psi:F\into G$ such that $\psi(y_i)=y_i$, $1\leq i \leq n$, and $\|\psi\|\leq(1+\eta)$. Consider the $\Sigma$-$\br$-$\br$-operator given by the identity $I:\seebr\into \seebr$. Again, uniformity of $\alpha$ implies that
\[\abr(I\tens \psi  (u);  \eee, G) \leq (1+\eta)\,\abr(u;  \eee, F).\]
Finally,
\begin{align*}
\ab( u;  \xxx, Y) &\leq \abr(u;  \eee, G)\\
            &=      \abr(I\tens \psi  (u);  \eee, G)\\
        &\leq      (1+\eta)\,  \abr(u;  \eee, F)\\
        &\leq      (1+\eta)^2\,  \ab(u ;  \xxx, Y^{**})\\
\end{align*}
holds for all $\eta>0$.

(ii): First, suppose $\overline{\fhi}$ is bounded. Let $u\in\left(\xtxty,\ab\right)$. Then $\fhi(u)=\overline{\fhi}(u)$ and $\ab\left(u; \xxx, Y\right)=\ab\left(u; \xxx Y^{**}\right)$. Therefore, $\fhi$ is bounded and $\|\fhi\|\leq\|\overline{\fhi}\|$.

Conversely, suppose $\fhi$ is bounded. Fix $u$ in $\left(\xtxtydd,\ab\right)$ and let $\eta>0$. Since $\alpha$ is finitely generated there exist finite dimensional subspaces
$E_i$ and $F$ of $X_i$, $1\leq i \leq n$, and $Y^{**}$ respectively such that $\etetf$ contains $u$ and
\[\abr\left(u;  \eee, F\right)\leq(1+\eta)\,  \ab\left(u;  \xxx, Y^{**}\right).\]
Let $\sum_{i=1}^m x_i^1\tens\dots\tens x_i^n\tens y_i^{**}$ be a fixed representation of $u$ in $\etetf$. Now, we may apply the Principle of Local Reflexivity to $F\subset Y^{**}$ and $span\{f_\fhi(x_i^1\tens\dots\tens x_i^n)\}\subset Y^*$ to find a finite dimensional subspace $G\subset Y$ and an isomorphism $\psi:F\into G$ with $\|\psi\|\leq 1+\eta$ and $\lev f_\fhi(x_i^1\tens\dots\tens x_i^n) \,,\,  \psi(y_i^{**}) \rev=\lev y_i^{**}\,,\,
f_\fhi(x_i^1\tens\dots\tens x_i^n)\rev$ for all $1\leq i \leq n$.

Then $\fhi(x_i^1\tens\dots\tens x_i^n\tens \psi(y_i^{**}))=\overline{\fhi}(x_i^1\tens\dots\tens x_i^n\tens y_i^{**})$, hence $\fhi\circ (I\tens\psi)(u)=\overline{\fhi}(u)$. Finally,
\begin{align*}
|\overline{\fhi}(u)|&=|\fhi\circ (I\tens\psi)(u)|\\
         &\leq\|\fhi\|\;\ab\left(I\tens\psi(u);  \xxx, Y\right)\\
         &\leq\|\fhi\|\;\abr\left(I\tens\psi(u);  \eee, G\right)\\
         &\leq\|\fhi\|\;(1+\eta)\;\abr\left(u;  \eee, F\right)\\
         &\leq\|\fhi\|\;(1+\eta)^2\;\ab\left(u;  \xxx, Y^{**}\right)
\end{align*}
completes the proof.
\end{proof}

If in Theorem \ref{tndsideal} we take $n=1$, then we recover the original relation between tensor norms and ideals of linear operators. For this reason we say that the $\Sigma$-ideal $\ideal$ and the $\Sigma$-tensor norm on duals $\nu$ are \emph{associated} if $\Abeef$ is isometrically isomorphic to $\left(\Lbee\tens F,\vb\right)$ for all elections $(n,\eee, F, \beta)$ in $\finn$.

\begin{theorem}\label{rt}{\bf (Representation Theorem for Maximal $\Sigma$-ideals)} Let $\nu$ be a $\Sigma$-tensor norm on duals and $\ideal$ be the maximal $\Sigma$-ideal associated to $\nu$. Then
\begin{align*}
\left(\xtxty,\ab\right)^*&=\Abxxyd\\
\left(\xtxtyd,\ab\right)^*\cap\Lxxy&=\Abxxy
\end{align*}
for every election $\xxyb$ in $\ban$, where $\alpha$ is the finitely generated $\stn$ on spaces on the class $\ban$ defined by $\nu$.
\end{theorem}

\begin{proof}
First, we will prove the second equality. Given finite dimensional subspaces $E_i$ of $X_i$, $1\leq i \leq n$ and a finite codimensional subspace $L$ of $Y$ we have, by hypothesis, that
\begin{equation}\label{rt1}
\left(\ete\tens (Y/L),^*\abr\right)^*=\mathcal{A}^{\br}\left(\eee;Y/L\right)
\end{equation}
is a linear isometric isomorphism.

Let $T\in\Abxxy$ and $\eta>0$ fixed. Let $\fhi_T$ be the associated functional of $T$. For $u\in\left(\xtxtyd,\ab\right)$ there exist $E_i\in \mathcal{F}(X_i)$, $1\leq i \leq n$ and $F\in \mathcal{F}(Y^*)$ such that $u\in\left(\etetf,\abr\right)$ and $\abr(u;\eee,F)\leq(1+\eta)\,  \ab(u;\xxx, Y)$. The space $F$ defines a finite codimensional subspace $L$ of $Y$ such that $(Y/L)^*=F$ holds linearly and isometrically via $Q_L^*$. Then, \eqref{rt1} ensures 
\begin{align*}
|\fhi_T(u)|           &=       |\fhi_T\circ (I_{\eee}\tens Q_L^*)(u)|\\
                      &\leq      \|\fhi_T\circ (I_{\eee}\tens Q_L^*):\left(\ete\tens (Y/L)^*,\abr\right)\into \kk\|  \abr(u)\\
                            &=     A^{\br}(Q_L  f_T I_{\eee})  (1+\eta)  \ab(u;\xxx,Y)\\
                       &\leq      \Ab(T)  (1+\eta)  \ab(u;\xxx, Y).
\end{align*}
Hence, $\fhi_T$ is bounded and $\|\fhi_T\|\leq A(T)$.

For the converse inequality let $\fhi\in\left(\xtxtyd, \ab\right)^*$ such that its associated multilinear operator has range contained in $Y$. First, notice that
\[\sup\left\{\; A(Q_L f_{T_\fhi} I_{\eee}) \;|\; E_i\in\mathcal{F}(X_i)\;L\in\mathcal{CF}(Y) \;\right\}<\infty.\]
This is easy to see since
\begin{equation}\label{rt2}
A(Q_L f_{T_\fhi} I_{\eee})=\|\fhi\circ (I_{\eee}\tens Q_L^*)\|\leq\|\fhi\|
\end{equation}
holds for all $E_i\in\mathcal{F}(X_i)$, $1\leq i \leq n$ and $L\in\mathcal{CF}(Y)$. This means that $T_\fhi\in\Abxxy$. Actually, if we take suprema in \eqref{rt2} over all $E_i$ and $L$ we obtain, by maximality, that $A(T_\fhi)\leq\|\fhi\|$.

For the first equality let $T\in\Abxxyd$. We will prove that
\begin{eqnarray*}
\zeta_T:(\xtxty,\ab) &\into    &\kk\\
             \xxt\tens y &\mapsto & \lev  f_T(\xxt), y \rev
\end{eqnarray*}
is bounded. By Proposition \ref{aronbernerextension} this occurs exactly when its canonical extension $\overline{\zeta_T}:(\xtxtydd,\ab)\into \kk$ is bounded. We just have proved that the functional $\fhi_T:(\xtxtydd,\ab)\into \kk$ is bounded and $\|\fhi_T\|=\Ab(T)$. But
\[\fhi_T(\xxt\tens y^{**}) =\lev y^{**} \,,\, f_T(\xxt)\rev =\overline{\zeta_T}(\xxt\tens y^{**})\]
asserts that $\fhi_T=\overline{\zeta_T}$. Hence $\zeta_T$ is bounded and $\|\zeta_T\|=\Ab(T)$.

Conversely, let $\fhi\in\left(\xtxty,\ab\right)^*$. Consider the associated multilinear operators $T_\fhi:\xpx\into Y^*$ and $T_{\overline{\fhi}}:\xpx\into Y^{***}$ of $\fhi$ and $\overline{\fhi}$. The definition of the canonical extension implies that $\lev T_{\overline{\fhi}}(\xxt)  \,,\, y^{**} \rev = \lev K_{Y^*}T_{\fhi}(\xxt)\,,\,y^{**}\rev$ for all $x_i^*\in X_i$, $1 \leq i \leq n$ and $y^{**}\in Y^{**}$. This means that $T_{\overline{\fhi}}$ has range contained in $Y^*$ and $T_{\overline{\fhi}}=K_{Y^*}T_\fhi$. Now, Proposition~\ref{aronbernerextension} implies
\[\overline{\fhi}\in\left(\xtxtydd,\ab\right)^*\cap\mathcal{L}\left(\xxx;Y^*\right) =\Abxxyd.\]
Finally, $K_{Y^*} T_\fhi\in\Abxxyddd$ asserts that $T_\fhi\in\Abxxyd$ and
\[\Ab(T_\fhi)=\Ab(K_{Y^*}T_{\fhi})=\|\overline{\fhi}\|=\|\fhi\|.\]
\end{proof}

In the theory of ideals of linear operators (and even of multilinear operators) it is common to have an isometry of the form $(X\tens_\alpha Y)^*=\mathcal{A}(X;Y^*)$ for all Banach spaces $X$ and $Y$. From this, it is easy to prove that the ideal is maximal if $\alpha$ is finitely generated, a condition which is easier to check than maximality. We finish this section by presenting a criterion for maximality of $\Sigma$-ideals. First, we prove a preliminary result.

\begin{proposition}\label{maximalhull}
Let $\ideal$ be a $\Sigma$-ideal on the class $\finn$. Define for every multilinear operator $\Txxy$
\[A^{max,\beta}(T):=\sup  A^{\br}(Q_L f_T I_{\eee}:\epe\into Y/L)\leq\infty,\]
where the suprema is taken over all $E_i\in\mathcal{F}(X_i)$, $1\leq i\leq n$, and $L\in\mathcal{CF}(Y)$.
Also define
\[\mathcal{A}^{max,\beta}(\xxx,Y):=\{\; T:\xpx\into Y \;|\; A^{max,\beta}(T)<\infty  \;\}.\]
Then the pair $[\mathcal{A}^{max},A^{max}]$ is a maximal $\Sigma$-ideal on the class $\ban$.
\end{proposition}

\begin{proof}
For shorten the article we only prove the more delicate issues of this proposition. That is we only prove the ideal property and that $A^{max,\beta}$ give place to a Banach space.

Let $(m,\zzzm, W, \theta)$ be an election in $\ban$. Consider the composition
\[\begin{array}{c}
\xymatrix{
\szzm\ar[r]^-\fr    &\sxx\ar[r]^-{f_T}  &Y\ar[r]^S                   & W\\
}
\end{array}\]
where $S$ is a bounded linear operator, $T$ is in $\mathcal{A}^{max,\beta}(\xxx;Y)$ and $\fr$ is a $\Sigma$-$\beta$-$\theta$-operator associated with $R$. Let $M_i$ be a finite dimensional subspace of $Z_i$, $1\leq i \leq m$ and $G$ be a finite codimensional subspace of $W$. Since $\fr(\szzm)$ is contained in $\sxx$ and $M_1\tens\cdots\tens M_m$ is a finite dimensional space then there exist finite dimensional subspaces $E_i$ of $X_i$, $1\leq i \leq n$ such that $\fr I_{M_1,\dots, M_m}(\Sigma_{M_1,\dots, M_m})\subset\Sigma_{\eee}$. Set $L=Ker(Q_GS)\in\mathcal{CF}(Y)$. Consider the commutative diagram
\[\begin{array}{c}
\xymatrix{
  \szzm\ar[drr]_{\fr I_{M_1,\dots, M_m}}\ar[rr]^{\fr}   &&\sxx\ar[rr]^{f_T}   && Y\ar[d]_{Q_L}\ar[r]^S   &W\ar[d]^{Q_G}\\
   \Sigma_{M_1,\dots, M_m}\ar[u]^{I_{M_1,\dots, M_m}}    &&\Sigma_{\eee}\ar[u]_{I_{\eee}}  &&Y/L\ar[r]_B  &W/G\\
}
\end{array}.\]
Then
\begin{align*}
A^{\theta|}\left(Q_G(Sf_T\fr)I_{\mmm}\right)   &=    A^{\theta|}(B(Q_L f_T I_{\eee})(\fr I_{\mmm}))\\
                &\leq    \|S\|  A^{\br}(Q_L f_T I_{\eee})  \|\rlin\|\\
                &\leq  \|S\|  A^{max,\beta}(T)  \|\rlin\|.
\end{align*}
After taking suprema over all $M_i$, $1\leq i \leq m$ and $G$ as above we have
\[A^{max}(Sf_T R)\leq\|S\|  A^{max,\beta}(T)  \|\rlin\|.\]

To see that $A^{max,\beta}$ is a complete norm let $(T^k)_{k=1}^\infty$ be a sequence in $\mathcal{A}^{max,\beta}(\xxx;Y)$ such that $\sum_{k=1}^\infty A^{max,\beta}(T^k)<\infty$. Define $T=\sum_{k=1}^\infty T^k:\xpx\into Y$. Notice that
\[ \sum\limits_{k=1}^\infty \|f_{T^k}(p)\| \leq   \sum\limits_{k=1}^\infty \Lb(f_{T^k}) \leq   \sum\limits_{k=1}^\infty A^{max,\beta}(T^k).\]
Hence, $\sum_{k=1}^\infty f_{T^k}(p)$ converges in $Y$ for every $p\in\sxx$ since $Y$ is a Banach space. Moreover, $\sum_{k=1}^\infty \Lb(f_{T^k}) \leq   \sum_{k=1}^\infty A^{max,\beta}(T^k)$ also implies that $f_T$ is a Lipschitz function on $\sxx$. That is, $T$ is a bounded multilinear operator. Finally, by definition of $f_T$, the triangle inequality and definition of $A^{max,\beta}$ we have that
\[A^\br(Q_Lf_T I_{\eee})  \leq    \sum\limits_{k=1}^\infty A^\br( Q_L f_{T^k} I_{\eee}) \leq     \sum\limits_{k=1}^\infty A^{max,\beta}(T^k)\]
holds for all $E_i$,$1\leq i \leq n$ and $L$.
\end{proof}

Given a $\Sigma$-ideal $\ideal$ on the class $\ban$ then $[\mathcal{A}^{max},A^{max}]$ is called the \emph{maximal hull} of $\ideal$. By definition, $\ideal$ and its maximal hull coincide on the class $\finn$ and $A^{max,\beta}(T)\leq \Ab(T)$ for all multilinear operator $T$. In this terms, a $\Sigma$-ideal is maximal if it coincides with its maximal hull.

\begin{proposition}\label{criterionmaximal}
Let $\ideal$ be a $\Sigma$-ideal. Suppose there exist a finitely generated $\Sigma$-tensor norm on spaces $\alpha$ such that for any election of Banach spaces $\xxyb$ we have that
\begin{equation}\label{criterion}
\left(\xtxty,\ab\right)^*=\Abxxyd
\end{equation}
holds linearly and isometrically. Then, $\ideal$ is maximal.
\end{proposition}

\begin{proof}
Theorem~\ref{tndsideal} ensures that $\ideal$ is a $\Sigma$-ideal on the class $\finn$. Proposition~\ref{maximalhull} tells us that its maximal hull is a $\Sigma$-ideal on $\ban$. The RT combined with (\ref{criterion}) asserts that
\begin{equation}\label{criterion1}
\Abxxyd=\mathcal{A}^{max,\beta}\left(\xxx;Y^*\right)
\end{equation}
holds linearly and isometrically for all elections $\xxyb$ in $\ban$. Another application of the RT combined with (\ref{criterion1}) leads us to
\[\mathcal{A}^{max,\beta}\left(\xxx;Y\right)=\mathcal{A}^\beta\left(\xxx;Y^{**}\right)\cap \Lxxy.\]
Finally, for $T$ in $\mathcal{A}^{max,\beta}\left(\xxx;Y\right)$ we have
\[\Ab(T)=\Ab(K_Y^{-1}K_YT)\leq \Ab(K_YT)= A^{max,\beta}(T).\]
The proof is complete since the converse inequality $A^{max,\beta}(T)\leq \Ab(T)$ is always true.
\end{proof}

\begin{example}
The following duality relation holds
\begin{align*}
(\xtxty,\pi^\beta)^*&=\mathcal{L}^\beta_{Lip}(\xxx;Y^*)\\
(\xtxty,\gamma_2^\beta)^*&=\Gamma^\beta(\xxx;Y^*)\\
(\xtxty,\alpha_{p^*,q^*}^\beta)^*&=\mathcal{D}_{p,q}^\beta(\xxx;Y^*)\\
(\xtxty,d_p^\beta)^*&=\Pi_p^{Lip,\beta}(\xxx;Y^*)\\
\end{align*}
The first equality follows easily adapting the linear case; the second is an adaptation of \cite[Theorem 4.4]{fernandez-unzueta18b}; the third is \cite[Theorem 4.7]{fernandez-unzueta18a}; and the fourth is \cite[Theorem 2.26]{angulo10}.
\end{example}

According to Proposition \ref{regular}, in the presence of maximality we always may assume that multilinear operators under study have range contained in a dual space. Combining the Representation and Duality Theorems we have that if a maximal $\Sigma$-ideal $\ideal$ is associated with the cofinitely generated $\Sigma$-tensor norm on duals $\nu$ then

\begin{equation}\label{conclusion}
\begin{array}{c}
\xymatrix{
(\xtxty,\ab)^*  \ar[r] &    \Abxxyd           \\
(\Lbxx\tens Y^*,\vb) \ar[r]\ar@{^(->}[u]       &  (\mathcal{F}^\beta(\xxx;Y^*),\Ab)\ar@{^(->}[u] \\
}
\end{array}
\end{equation}
commutes, where the horizontal arrows are isomorphic isometries, the vertical ones are linear isometries and $\alpha$ is the finitely generated $\Sigma$-tensor norm on spaces defined by $\nu$ (see comment after Definition \ref{ideal}).

\subsubsection{Basic Results}

As can be seen, the most technical difficulty of the notion of $\Sigma$-ideals of multilinear operators is the ideal property since this involves all reasonable crossnorms $\beta$. Next, we explore some consequences of this consideration.

\begin{proposition}
Let $\alpha$ be a $\Sigma$-tensor norms on spaces and $\ideal$ be a $\Sigma$-ideal. Let $n$ be a positive integer, $\xxx$ and $Y$ be Banach spaces. Let $\beta$ and $\theta$ two reasonable crossnorms on $\xtx$ such that there exists $\lambda$ in $\kk$ such that $\beta\leq \lambda\theta$. Then
\begin{itemize}
\item [(i)] $\alpha^\beta (u)\leq \lambda \alpha^\theta(u)$ for all $u\in\xtxty$.
\item [(ii)] $A^\theta(T)\leq \lambda A^\beta(T)$ for all $T\in\Abxxy$.
\item [(iii)] The inclusions $\mathcal{A}^\varepsilon (\xxx; Y)\subset \mathcal{A}^\beta (\xxx; Y) \subset \mathcal{A}^\pi (\xxx; Y)$ are continuous.
\item [(iv)] The inclusions $(\xtxty; \alpha^\pi)\subset (\xtxty; \alpha^\beta) \subset (\xtxty; \alpha^\varepsilon)$ are continuous.
\end{itemize}
\end{proposition}

\begin{corollary}
Let $\alpha$ be a $\Sigma$-tensor norm on spaces and $\ideal$ be a $\Sigma$-ideal. Let $\xxx$ and $Y$ be Banach spaces. If $\beta$ and $\theta$ are equivalent reasonable crossnorms on $\xtx$, then
\begin{itemize}
\item[(i)] The components $\Abxxy$ and $\mathcal{A}^\theta(\xxx;Y)$ are isomorphic.

\item[(ii)] The products $(\xtxty,\ab)$ and $(\xtxty,\alpha^\theta)$ are isomorphic.
\end{itemize}
\end{corollary}

In some cases it is not necessary to consider the restriction of the reasonable crossnorm $\beta$.

\begin{corollary}
Let $\alpha$ be a finitely generated $\Sigma$-tensor norm on spaces and $\ideal$ be a maximal $\Sigma$-ideal. Let $(n, \xxx, Y, \beta)$ be and election in $\ban$ where $\beta$ is and injective tensor norm. Then
\begin{itemize}
\item  [(i)]  $\Ab(T)=\sup \{\Ab(Q_L f_T I_{\eee}) | E_i\in \mathcal{F}(X_i), F\in \mathcal{CF}(Y)\}$.
\item [(ii)]  $\ab(u; \xxx,Y)=\inf\{ \ab(u; \eee,F) | E_i\in \mathcal{F}(X_i), F\in \mathcal{F}(Y)\}$.
\end{itemize}
\end{corollary}

The previous corollary is also true if we replace the Banach spaces $X_i$ by Hilbert spaces $H_i$ and $\beta$ by the reasonable crossnorm $\|\cdot\|_2$ (the norm which makes $H_1\tens\dots\tens H_n$ a (pre) Hilbert space, see \cite[Section 2.6]{kadison83} for details).

In every $\Sigma$-ideal $\ideal$ there exist cases where components coincide. Recall from \cite{pisier83} that Pisier prove the existence of Banach spaces $X$ such that $(X\hat{\tens} X,\varepsilon)$ is isomorphic to $(X\hat{\tens} X,\pi)$. For these spaces $X$ we have that every component $\mathcal{A}^\beta(X,X;Y)$ is isomorphic to $\mathcal{A}^\pi(X,X;Y)$ for all Banach space $Y$.

We finish the article presenting some consequences for the case $n=2$. According to Grothendieck there exist fourteen natural tensor norms, this is, tensor norms obtained from $\pi$ after the application of a sequence, in any order, of the operations: taking dual norm, injective associated, projective associated and transposing (see \cite[pp. 184]{ryan02}). In the following diagram we exhibit the consequence of this observation. In it, $\mathcal{A}^\beta$ stands for $\mathcal{A}^\beta(X_1,X_2;Y)$ and an arrow from $\mathcal{A}^\beta$ to $\mathcal{A}^\theta$ indicates the existence of a constant $\lambda$ such that $ A^\theta (T)\leq \lambda \Ab(T)$

\begin{equation*}
\begin{array}{c}
\xymatrix{
                                                                 &      \mathcal{A}^\pi &   \\
  \mathcal{A}^{\backslash(/\pi)}\ar[ur]             &         &   \mathcal{A}^{(\pi\backslash)/}\ar[ul]          \\
  \mathcal{A}^{/\pi}\ar[u]  & \mathcal{A}^{\backslash\varepsilon/}\ar[uu]\ar[ul]\ar[ur] &  \mathcal{A}^{\pi\backslash}\ar[u]          \\
  \mathcal{A}^{/(\backslash\varepsilon/)}\ar[u]\ar[ur]&         &   \mathcal{A}^{(\backslash\varepsilon/)\backslash}\ar[u]\ar[ul]\\
  \mathcal{A}^{\varepsilon/}\ar[u]&      \mathcal{A}^{/\pi\backslash}\ar[uu]\ar[ul]\ar[ur] &   \mathcal{A}^{\backslash\varepsilon}\ar[u]          \\
  \mathcal{A}^{(\varepsilon/)\backslash}\ar[u]\ar[ur]&         &   \mathcal{A}^{/(\backslash\varepsilon)}\ar[u]\ar[ul]          \\
                                                                 &      \mathcal{A}^\varepsilon\ar[ul]\ar[uu]\ar[ur] &  
  }
\end{array}.
\end{equation*}


\bibliographystyle{plainurl}

\end{document}